\documentclass[11pt]{article}
\usepackage[utf8]{inputenc}
\usepackage{geometry}
\usepackage{enumitem}
\setlist{nosep}

\usepackage{xspace}

\usepackage[skip=4pt, indent=30pt]{parskip}
\usepackage{xcolor}
\usepackage{amsfonts}
\usepackage{bbold} % to be able to do mathbb{1}
\usepackage{amsmath}
\usepackage{esvect} % overhead arrow
\usepackage{ amssymb }
\usepackage{ mathrsfs } % for pretty callygraphy letters
\DeclareFixedFont{\titlefont}{T1}{ppl}{b}{n}{0.5in}
\usepackage[T1]{fontenc}
\usepackage{libertine} % for bolded textsc
\usepackage{subcaption}
\usepackage{physics} % for \order

\RequirePackage{tikz}
\usepackage{pgf}
\usepackage{graphicx}

\usepackage[ruled, linesnumbered]{algorithm2e}

\graphicspath{{}{timeline/}}

\usepackage{amsthm}
\usepackage{thmtools}

\usepackage{hyperref}
\usepackage{cleveref}

\declaretheorem{theorem}
\declaretheorem{definition}

\declaretheorem{lemma}
\declaretheorem{assumption}

\renewcommand{\epsilon}{\varepsilon}
\renewcommand{\phi}{\varphi}
\renewcommand{\geq}{\geqslant}
\renewcommand{\leq}{\leqslant}

\newcommand{\R}{\mathbb{R}}
\newcommand{\N}{\mathbb{N}}
\newcommand{\E}{\mathbb{E}}
\newcommand{\pr}{\mathbb{P}}
\newcommand{\one}{\mathbb{1}}
\newcommand{\cX}{\mathcal{X}}
\newcommand{\XP}{\mathcal{X}_{\mbox{\tiny Pre}}}
\newcommand{\XNP}{\mathcal{X}_{\mbox{\tiny NonPre}}}

\usetikzlibrary{decorations.pathreplacing}
\usetikzlibrary{positioning, arrows.meta}
\usetikzlibrary{shapes.geometric, calc}

\begin{document}

\title{Scheduling jobs with unknown size distribution in a M/G/1 queue: the shifted empirical Gittins}

\author{
  {Nicolas Gast$^1$, Bruno {Gaujal$^1$} and Adrien Obrecht$^2$}, \\
  {1. Univ.\ Grenoble Alpes, Inria, CNRS, LIG, F-38000 Grenoble, France.}\\
  {2. ENS de Lyon, Lyon, France.}
}

\maketitle

\abstract{In this paper we consider a M/G/1 queue for which we want to minimize the expected response time. We show  how to compute indices from $n$ samples of the job size distribution such that the corresponding index policy  is asymptotically optimal as $n$ grows.
  This construction is based on a discretization of the  bounded support of the job size distribution and a  shift of the samples to their nearest discrete point to the right.
  We  show that the Gittins index  of the empirical distribution of these shifted samples is close to the Gittins index  of the original distribution. This translates to the asymptotic optimality of the corresponding index policy for minimizing the expected response time. Numerical comparison with other approaches further confirm  the efficiency of our approach.}

\section{Introduction}

In this paper, we focus on the problem of scheduling jobs  in a M/G/1 queue to minimize their  mean response time under the following conditions.
\begin{itemize}
\item All jobs have the same size distribution (technical conditions on the distribution will be  discussed later),
\item jobs arrive according to a Poisson process,
  \item jobs are executed  by a single server that can switch from one job to another instantaneously.
  \end{itemize}

 This problem has drawn a lot of interest recently for several reasons. Indeed, the solutions under different information levels on the job sizes are quite elegant and they are enlightening on what features are crucial to obtain small response times.
 \begin{itemize}
 \item 
 In the {\it clairvoyant case} (jobs announce their size when they arrive in the
 queue), the optimal scheduling policy is SRPT (Smallest Remaining Processing Time) (see \cite{SRPT} for a proof).
 
 \item In the {\it non-clairvoyant case}, (the size of the job is only revealed at completion time), the optimal scheduling policy is the famous Gittins index policy (see \cite {Scully2021} for a detailed proof) when the  distribution of the job sizes is known in advance.

\item In this paper,  we consider the {\it non clairoyant black box case}: The size of the job is only known at completion time and the distribution of the job sizes is also unknown.

\end{itemize}

Our approach to solve this problem is  of the ``{\it explore then commit}'' type: We first  observe $n$ samples of job sizes and infer a scheduling policy that becomes asymptotically optimal as the number of samples $n$ goes to infinity.
 The main challenge comes from the fact that the empirical distribution induced by  $n$ samples  is made of $n$ atoms all this equal probability. This is quite distinct from the actual distribution of jobs when it is continuous  w.r.t. the Lebesgue measure.
 To overcome this difficulty, one may infer an empirical density from the $n$ samples, for example using a Gaussian kernel (see \cite{Rudemo} for example) and compute the Gittins indexes for this estimated density. Unfortunately, the special shape of the Gittins indexes is highly sensitive to the distribution and this direct approach is not very efficient, as seen in our numerical tests (Section \ref{sec:numerical}).
 In this paper, we use a different approach by introducing a transformation on the samples that we call the ``right shift'' that is compatible with the way index policies work and this is a key ingredient in the proof of our main result, namely asymptotic optimality of the learned index policy.

When finalizing this manuscript, we found that the paper \cite{ramakrishna2026empirical} has just been published.  This paper provides a convergence analysis of the empirical Gittins policy without the need to resort to a shifted index.
 
 \section{Problem description}

We consider jobs that arrive in an online fashion on a queue with a single processing node,
with  an associated execution time (denoted $X_i$ for job number $i$). At each given point, the processor
can execute a single job at speed one. Once a
job $i$ has been executed for  $X_i$ units of time, it is considered finished and leaves the system.
We denote by $T_i$ the response time of job $i$,  namely the difference between its completion time and its arrival time.  Our
goal will  to minimize the expected mean response time of the jobs.

% \begin{definition}[Shortest Remaining Processing Time (SRPT)]
%     The SRPT policy schedules the job with the shortest remaining processing
%     time.
% \end{definition}

% % TODO refs
% \begin{theorem}
%     SRPT is a policy that minimises the expected mean response time.
% \end{theorem}

% When the execution time is known for each job,  the very simple SRPT policy
% is optimal. However most of the time, we don't have such precise information
% about about a job. We consider the model where we only know the age of a job,
% i.e. the total cumulated time it has been executed.

We also suppose that jobs
arrive following a Poisson arrival process of rate $\lambda$ and that their
size $X$ follows a  distribution $\phi$. %  with finite mean denoted $1/\mu$. The
% utilization $ \lambda / \mu$ is smaller than 1: $ \rho := \lambda / \mu < 1$, so
% that the queue is stable. We also assume that the job size has finite second moment: $\E(X^2) < \infty$.
We denote by  $f$ the probability density  function (PDF), when it exists,  by  $F$ the cumulative distribution function (CDF) and by $\bar{F}$ the complementary  CDF ($\bar{F} := 1-F$) of the job size.
 % In general $\phi$ does not have a density (it may contain atoms) and $F$ is not continuous.
The processor works at speed one and can only execute one job at a time. The job being currently executed is called the {\it active}  job.
The {\it age} $a$ of a job is the total amount of time the processor has been executing the job.
The set of ages of a job is partitioned into two sets, $\XP$ and $\XNP$, where the job can be preempted and cannot be preempted, respectively.

\begin{definition}[Gittins Index, Gittins index policy]\ \\
   The Gittins index function at age $x$ of a job with  size
   measure $\phi$ and remaining probability function $\bar{F}$ is

    \begin{equation}
      \label{eq:Gittins}
      \nu(x) := \sup_{y > x}\frac{\pr (X \leq y | X>x)}{\E(\min(X,y) -x | X>x)} =  \sup_{y > x} \frac{\int_{x}^{y}  d\phi(u)}{\int_{x}^{y} \bar{F}(u) du}.
    \end{equation}
    The corresponding Gittins index policy works as follows:
    Whenever the  age $x$ of the current active job is preemptable ($x\in \XP$), the Gittins index policy activates one  job with the maximal current Gittins index.
  \end{definition}

  The Gittins index function is right-continuous, but may not be continuous.
Note that when $\phi$ contains an atom (say at age $x_0$,) then   $\lim_{x\to x_0, x< x_0}  \nu(x) = +\infty$, and   $\lim_{x\to x_0, x >  x_0}  \nu(x) = \nu(x_0)$. 
The Gittins index policy is known to be optimal for minimizing the expected response time of the jobs when the only information available on the jobs  is  their ages. Several versions of this result exist in the litterature. The most recent and general proof is given in \cite{Scully2021}. This result  is very strong: for any job in the queue, the Gittins index of a job is a function that only depends on its age (and not on the other jobs, nor on the preemption set) and yet,  it  completely characterizes which job to schedule first in an optimal fashion.  However, the indices can only be computed when the distribution $\phi$ is known.  In this work we show how to define an approximate value of the Gittins indices while only having samples from $\phi$.

\subsection{Contributiuon and positioning}

In this paper, we propose a way to estimate the Gittins index function based on $n$ independent job samples and show that the corresponding Gittins index policy becomes optimal when $n$ goes to infinity.
We will leverage the definition of Gittins index to construct an {\it ad hoc}  estimated distribution function and {\it ad hoc} estimated  indices that make the corresponding policy close to the Gittins index policy.

Our approach works in two stages: we first fix the set of ages where jobs can be preempted as $\XP^{\Delta,\epsilon}:= \{ \frac{i}{\Delta} < 1-\epsilon,  i \in \N, \}$, where  $\Delta \in \N$ is a free parameter and $\epsilon \in (0,1)$.
% Here $\xend$ is the special state reached by the job when its age reaches its size: $\xend  \in \XP^{\Delta,\epsilon}$ just means that when the job is completed, the job can (and must)  be preempted.
Therefore, in our first framework,  the scheduling decisions are reexamined every $1/\Delta$ units of time (here the processor has speed one) or when the active job finishes. % One should see $\Delta$ as a large integer with respect to the typical size of the jobs, so that the opportunities to preempt the active job  are plenty before the job is completed, although this is not necessary for our results to hold.
We show (Section \ref{ssec:perf}) that when $n$ grows, the estimated Gittins index policy becomes optimal under this $\Delta$-preemption case.

While this result has some interest by its own, the second step (Section \ref{ssec:no-preemption}) shows how to let  $\Delta$ grow as $n$ goes to infinity so that
the preemption set becomes dense in the set of ages. n this case,  the  estimated Gittins index policy becomes optimal without any restriction on the preemptive set.

\section{Discretization and right shift}

The main idea that underlies our estimated index is based on the construction of a shifted distribution. In this section, we consider any distribution with a support included in $\R_+$ the set of non-negative real numbers and we discretize it using an integer   parameter $\Delta \in \N$.

\subsection{Right-shifted distribution}

% We assume that we have $n$ samples taken independently from the density distribution
% $f$, and we want to compute an approximation of the Gittins indices for $f$, such
% that our resulting policy performs close to the optimal Gittins policy (that has
% perfect information about $f$). The goal is the following: as the number of
% samples increases, our computed indices should get  closer tand closer to the real indices, and then
% as the indices get closer, the resulting policy performs within a good
% approximation of the optimal gittins policy.

% We will first add a few strengthening assumptions about $f$, to make the
% analysis easier:

% \begin{itemize}
%     \item The support of $f$ is $[0, 1]$.
% %    \item $f(0)$ is finite.  NEEDED? 
%     \item $\forall t \in [0, 1], f(t) \geq \ell$.
%     \item $f$ is a $L$-Lipschitz continuous function.
% \end{itemize}

We first define a shifted {discretization} of the measure $\phi$: we denote by $\phi_\Delta$ the
shifted version of $\phi$,  where the new distribution is atomic with mass at points  that are multiples of $1/\Delta$.
We will use the handy notations $\lfloor t \rfloor_\Delta$ and $\lceil  t \rceil_\Delta$ to be the greatest (resp. smallest) multiple of $1/\Delta$ smaller (resp. greater) than $t$;
\begin{align*}
    \lfloor t \rfloor_\Delta := \frac{\lfloor t \Delta\rfloor}{\Delta},\text{ and }  \lceil t \rceil_\Delta := \frac{\lceil t \Delta\rceil}{\Delta}.
\end{align*}
We also use a strict version of the floor operator: 
   \[ \lfloor \!\lfloor t \rfloor \!\rfloor_\Delta := \frac{\lceil t \Delta\rceil -1}{\Delta}.\]

\begin{definition}[shifted discrete distribution]
    % Let $\Delta \in \mathbb{N}$. We define the discrete random variable
    % $X_\Delta$ by:

    % $$\forall i \in [1, \Delta], \mathbb{P}\left(X_\Delta = i \frac{1}{\Delta}\right) =
    % \mathbb{P}\left(\frac{i-1}{\Delta} < X \leq \frac{i}{\Delta} \right)$$
Let $\Delta \in \N$ and let $\phi$ be any positive measure on $\R$.
  The measure  $\phi_\Delta$ is defined by its value on any interval $[u,v]$:
  \[ \phi_\Delta(\big[u,v\big]) := \phi(\big[\lfloor \! \lfloor u \rfloor \! \rfloor_\Delta,  \lfloor v \rfloor_\Delta \big] ). \]
      \end{definition}

A more operational view of $\phi_{\Delta}$ is $\forall t$, 
\begin{align*}
         \phi_\Delta(t) 
        &= \sum_{i=1}^{\infty} \delta_{i/\Delta}(t) \phi\left(\left(\frac{i-1}{\Delta}, \frac{i}{\Delta}\right]\right) \\
          &= \sum_{i=1}^{\infty} \delta_{i/\Delta}(t) 
        \int_{\frac{(i-1)^+}{\Delta}}^{\frac{i}{\Delta}} d\phi(s) \\
  %      &= \sum_{i=1}^{\Delta} \delta_{i/\Delta}(t)
   %     x_\Delta^i 
    \end{align*}
where the integral $  \int_{\frac{(i-1)^+}{\Delta}}^{\frac{i}{\Delta}} d\phi(s)$ excludes any atom at $(i-1)/\Delta$ and includes any atom at $i/\Delta$. For a given $x\in\R$, $\delta_x$ denotes the Dirac distribution in $x$.

We can bound the cumulated distribution function $F_\Delta $ of this shifted distribution by the original cumulated distribution $F$.

\begin{lemma}
    $\forall {\Delta} \in \mathbb{N}, \forall t \in \R_+ F_\Delta(t) \leq F(t)$.
\end{lemma}

\begin{proof}
  This is a quite straightforward property: Let ${\Delta} \in \N$ and $t \in \R_+$.
  \begin{align*}
        F_\Delta(t) 
        &= \int_{0}^{t} d\phi_\Delta(t)\\
        &= \int_{0}^{t} \left( \sum_{i=1}^{\Delta} \delta_{i/\Delta}(t)  \int_{\frac{i-1}{\Delta}}^{\frac{i}{\Delta}} d\phi(s) \right) dt \\
        &= \sum_{i=1}^{\Delta} \int_{0}^{t}  \left( \delta_{i/\Delta} (t)  \int_{\frac{i-1}{\Delta}}^{\frac{i}{\Delta}} d\phi(s) \right) dt\\
   %     &= \sum_{i=1}^{\frac{\lfloor t \rfloor_\Delta}{\frac{1}{\Delta}}} x_\Delta^i \\
        &= \int_{0}^{\lfloor t \rfloor_\Delta} \phi(s)  \\
        &\leq \int_{0}^{t} d\phi(s)  = F(t).
    \end{align*}
\end{proof}

\begin{figure}[hbtp]
    \centering
    \begin{tabular}{@{}c@{}c@{}}
        \resizebox{0.5\textwidth}{!}{\input{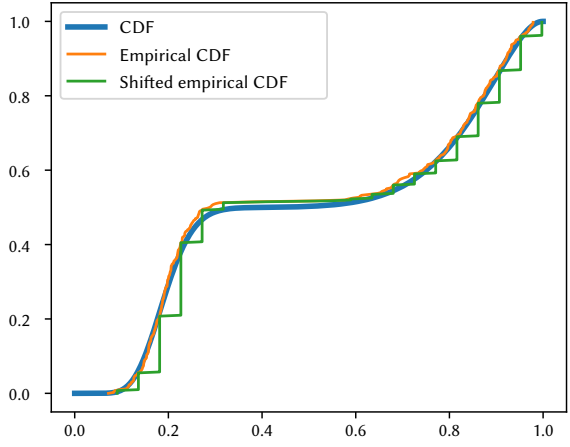}}
        &\resizebox{0.5\textwidth}{!}{\input{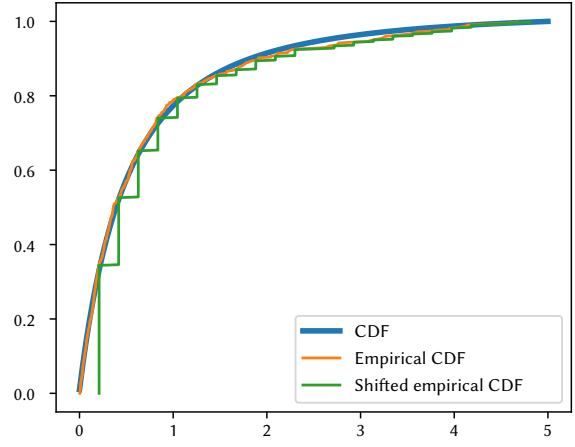}}\\
        (a) The distribution $\beta_{12, 50} + \beta_{10, 2}$ & (b) The truncated Pareto(2) distribution.
    \end{tabular}
    \caption{\label{fig:cdf} We plot the cumulative distribution $F$, the empirical distribution $\widehat{F_n}$ with $n=400$ samples and the shifted distribution $\widehat{F_{n,\Delta}}$, with $n=400$ and $\Delta = 20$. Each panel correspond to a specific initial distribution $F$.}
\end{figure}

We illustrate the shifted distribution in Figure~\ref{fig:cdf} where we display the three cumulative function ($F$, $\widehat{F_n}$ and $\widehat{F_{n,\Delta}}$) for types of job sizes. The distribution are obtained with $200$ \emph{i.i.d.} samples from $F$, and $\widehat{F_{n,\Delta}}$ is discretized by using $\Delta = \sqrt{n}$. In Panel~(a), the jobs are generated according to a distribution $F$ that is a mixture of two beta distribution with parameters $(12,50)$ and $(10,2)$. We call this distribution $\beta({12, 50} + \beta{10, 2}$.  In Panel~(b), the jobs are generated according to a distribution $F$ that is a Pareto($2$) distribution truncated to that all job sizes are between $0$ and $5$. 

\subsection{The shifted empirical index}

In this section we  define the shifted empirical index  from $n$ independent samples of the job size, $X_1,\ldots, X_n$.
Here, $\Delta$ is an arbitrary integer used for discretization. Scaling $n$ with respect to $\Delta$  will be done in the next section.
We define several index functions, using several distributions.
\begin{itemize}
\item 
$\nu$ is the Gittins index function of $f$, the distribution of the job size as defined in \eqref{eq:Gittins}
\item $\nu_\Delta$ is the Gittins index function of $f_\Delta$, the shifted discrete distribution.

\item Now, let  $\hat{\phi_n}$ be the empirical density distribution defined by the $n$ samples: $\hat{\phi_n} = \sum_{i=1}^n \frac{1}{n}\delta_{X_i}$. The function $\widehat{\nu_n}$ is the Gittins index function of $\hat{\phi_n}$, the empirical  distribution. 
  
\item Finally, let $\widehat{\phi_{n,\Delta}}$ be the shifted discrete version of $\hat{\phi_n}$.
  The {\it shifted empirical index} $\widehat{\nu_{n,\Delta}}$ is the Gittins index function of the distribution $\widehat{\phi_{n,\Delta}}$.

\end{itemize}

\begin{figure}[hbtp]
    \centering
    \begin{tabular}{cc}
        \resizebox{0.5\textwidth}{!}{\input{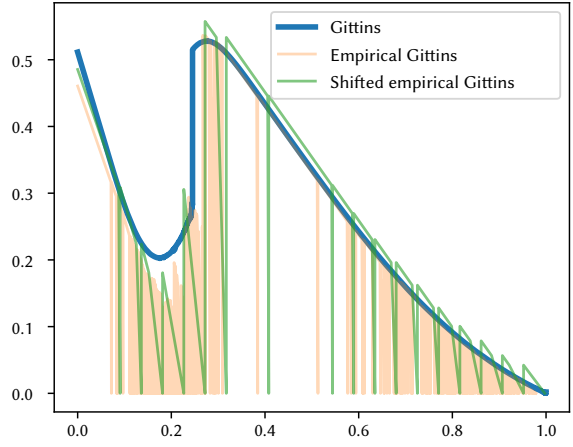}}
        &\resizebox{0.5\textwidth}{!}{\input{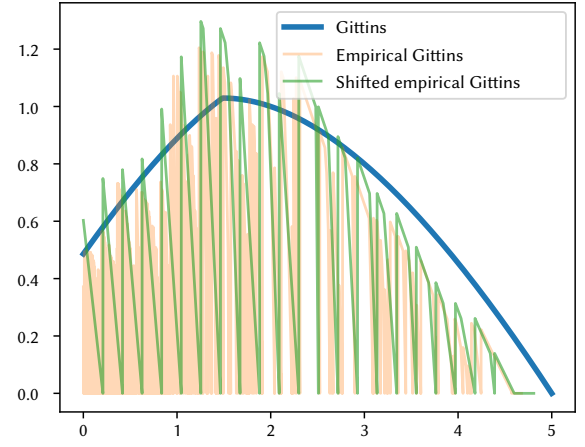}}\\
        (a) Distribution of jobs is $\beta_{12, 50} + \beta_{10, 2}$ & (b) Distribution of jobs is Pareto(2)
    \end{tabular}
    \caption{This figure displays the  rank $r$  (inverse of the index: $r := 1/\nu$), for the different index functions $\nu$, $\widehat{\nu_n}$ and $\widehat{\nu_{n,\Delta}}$ respective indices for the distributions  $F$,  $\widehat{F_n}$ and $\widehat{F_{n,\Delta}}$ given in the previous figure.}
    \label{fig:index}
\end{figure}

Note that the same shifted empirical index would be obtained if each sample $X_i$  is  shifted to the smallest multiple of $1/\Delta$ greater than $X_i$, instead of shifting the empirical distribution. We illustrate the three definition of indices in Figure~\ref{fig:index} where we diplay the inverse (sometimes called rank) of the Gittins index functions $\nu,\widehat{\nu_n}$ and $\widehat{\nu_{n,\Delta}}$ of the distributions displayed in Figure \ref{fig:cdf}. % are displayed in Figure \ref{fig:index}.

One could wonder how much of a negative impact the replacement of the real Gittins function by its shifted sample version can have on of the resulting scheduling policy. We will show that the impact is minimal, and decreases as a function of ${\Delta}$ in the next section. Actually shifting  the samples to the right is designed  to reduce the  impact on the policy. Here some high level intuition on why shifting samples to the right is a good idea. The right shift can be thought as overestimating the duration of jobs. Now if a job is selected to be executed at time $t$ because its index is large, then any premature completion due to its real size can only be a good thing for the response time. On the other hand, shifting the samples to the left would have a much bigger negative impact: think of a distribution with two atoms, one at $A$ and one at $1$, where $A$ is not a multiple of $1/\Delta$. Shifting  the samples from $A$ to the left would result in indices that become large just before time $A$ and become small again at time $A$. The corresponding index policy will stop the execution of the current  job before time $A$, missing an opportunity to complete the job (if the job is of size $A$). Similar observations were done in  \cite{moseley2025robust} to construct a robust scheduling for Gittins indices.

\section{Convergence of the shifted empirical index}

The goal of this section is to establish relations between the various Gittins index functions. We  show that the shifted empirical index is quite close to the real Gittins index when the number of samples grows. We make the following technical assumptions on the density $f$: 

\begin{assumption}\label{ass:1}
  The distribution $\phi$ has a density $f$ (no atoms) and has the following properties.
 \begin{itemize}
    \item The support of $f$ is $[0, 1]$.
%    \item $f(0)$ is finite.  NEEDED? 
    \item $\forall t \in [0, 1], f(t) \geq \ell$.
    \item $f$ is a $L$-Lipschitz continuous function.
    \end{itemize}
  \end{assumption}

 This assumption may look quite strong, although any distribution with finite support can be approached by a distribution satisfying this assumption. The finite support assumption is useful to get  simple closed form expressions for the bounds. We  believe it could be removed to the price of replacing uniform bounds by distribution dependent bounds. 
  
% We let $\nu_\Delta(\frac{i}{\Delta}) =
% \sup_{t'>\frac{i}{\Delta}} \frac{\int_{\frac{i}{\Delta}}^{t'} f_\Delta(s)ds}{\int_{\frac{i}{\Delta}}^{t'}
% \bar{F}_\Delta(s) ds}$ be the pushed gittins indices, we have the following
% lemma:

\begin{lemma}\label{lem:shifted}
    Let $\epsilon >0$ and $\Delta \in \N$. For all  $ i \in \mathbb{N}$ such that $\frac{i}{\Delta} < 1
    - \epsilon$, then:

    $$\nu(\frac{i}{\Delta}) \frac{1 -\frac{L}{\Delta \ell \epsilon}}{1 +
   \frac{L}{\Delta \ell \epsilon}}\leq \nu_\Delta(\frac{i}{\Delta}) \leq \nu(\frac{i}{\Delta}).$$
\end{lemma}

\begin{proof}
%    Let $\frac{1}{\Delta} \in \mathbb{R}$ and $i \in \mathbb{N}$.
    
    Right  inequality:

    \begin{align*}
        \nu_\Delta(\frac{i}{\Delta}) 
        &= \sup_{t'>\frac{i}{\Delta}} \frac{\int_{\frac{i}{\Delta}}^{t'}
        f_\Delta(s)ds}{\int_{\frac{i}{\Delta}}^{t'} \bar{F}_\Delta(s) ds} \\
        &\leq \sup_{t'> \frac{i}{\Delta}} \frac{\int_{\frac{i}{\Delta}}^{t'}
        \sum_{j=1}^\Delta \delta_{j/\Delta}(s)  \left( \int_{\frac{j-1}{\Delta}}^{\frac{j}{\Delta}} f(u) du \right) ds}{\int_{\frac{i}{\Delta}}^{t'} \bar{F}(s) ds} \\
        &\leq \sup_{t'>\frac{i}{\Delta}} \frac{\int_{\frac{i}{\Delta}}^{\lfloor t' \rfloor_\Delta} f(s)
        ds}{\int_{t}^{t'} \bar{F}(s) ds} \\
        &\leq \sup_{t'>\frac{i}{\Delta}} \frac{\int_{\frac{i}{\Delta}}^{t'} f(s) ds}{\int_{t}^{t'}
        \bar{F}(s) ds} \\
        &\leq \nu(\frac{i}{\Delta}). \\
    \end{align*}

    %TODO Lemma F_N <= F <= F_N + L/N

   Left  inequality:

    \begin{align*}
        \nu_\Delta(\frac{i}{\Delta}) 
        &= \sup_{t'> \frac{i}{\Delta}} \frac{\int_{\frac{i}{\Delta}}^{t'}
            \sum_{j=1}^{\Delta}
         \delta_{j/\Delta}(s)  \left( \int_{\frac{j-1}{\Delta}}^{\frac{j}{\Delta}} f(u) du \right) ds}{\int_{\frac{i}{\Delta}}^{t'} \bar{F}_\Delta(s) ds}
        \\
        &\geq \frac{\int_{\frac{i}{\Delta}}^{\lceil t^* \rceil_\Delta}
        \sum_{j=1}^{\Delta}
        \delta_{j/\Delta}(s)  \left( \int_{\frac{j-1}{\Delta}}^{\frac{j}{\Delta}} f(u) du \right) ds}{\int_{\frac{i}{\Delta}}^{\lceil t^*
            \rceil_\Delta}
        \bar{F}_\Delta(s) ds} \\
        &\geq \frac {
            \int_{\frac{i}{\Delta}}^{\lceil t^* \rceil_\Delta} f(s) ds
        }{
            \int_{\frac{i}{\Delta}}^{\lceil t^* \rceil_\Delta} \bar{F}_\Delta(s) ds
          },
    \end{align*}
    where $t^*$ achieves the supremum  in the first equality.
    Using $t^* \geq \frac{i}{\Delta} + \epsilon$, we get
    \begin{align*}
        &\geq \frac{
            \int_{\frac{i}{\Delta}}^{t^*} f(s) ds\left(1 - \frac{\int_{\lceil
            t^* \rceil_\Delta}^{t^*} f(s) ds}{\int_{\frac{i}{\Delta}}^{t^*}
            f(s)ds}\right)
        }{
            \int_{\frac{i}{\Delta}}^{\lceil t^* \rceil_\Delta} (\bar{F}(s) + L\frac{1}{\Delta}) ds
        } \\
        &\geq \frac{
            \int_{\frac{i}{\Delta}}^{t^*} f(s) ds\left(1 - \frac{\int_{\lceil
            t^* \rceil_\Delta}^{t^*} f(s) ds}{\int_{\frac{i}{\Delta}}^{t^*}
            f(s)ds}\right)
        }{
            \int_{\frac{i}{\Delta}}^{\lceil t^* \rceil_\Delta} \bar{F}(s) ds \left(1 +
            \frac{(\lceil t^* \rceil_\Delta - \frac{i}{\Delta})L \frac{1}{\Delta}}{
            \int_{\frac{i}{\Delta}}^{\lceil t^* \rceil_\Delta} \bar{F}(s) ds}\right)
        } \\
        &\geq \nu(\frac{i}{\Delta}) \frac{
            1 - \frac{\int_{\lceil t^* \rceil_\Delta}^{t^*} L
            ds}{\int_{\frac{i}{\Delta}}^{t^*} \ell ds}
        }{
            1 + \frac{(\lceil t^* \rceil_\Delta - \frac{i}{\Delta})L\frac{1}{\Delta}}{
            \int_{\frac{i}{\Delta}}^{\lceil t^* \rceil_\Delta} \epsilon \ell ds}
        } \\
        &\geq \nu(\frac{i}{\Delta}) \frac{
            1 -\frac{L}{\Delta \ell \epsilon}
        }{
            1 + \frac{L}{\Delta \ell \epsilon}
        }.
        %t^* < \frac{i}{\Delta} + \epsilon \\
        %&\geq \frac{
        %    \int_{\frac{i}{\Delta}}^{t^*} f(s) \left(1 - \frac{\int_{\lceil
        %    t^* \rceil_N}^{t^*} f(s) ds}{\int_{\frac{i}{\Delta}}^{t^*}
        %    f(s)ds}\right)
        %}{
        %    \int_{\frac{i}{\Delta}}^{\lceil t^* \rceil_N} (\bar{F}(s) + \frac{L}{N}) ds
        %} \\
    \end{align*}
\end{proof}

After proving that the shifted distribution gives Gittins indices close to the indices of the
real distribution, we now show that the sampled version of the  indices are close to
the shifted ones. Since both density distributions $ \widehat{f_{n, \Delta}}$ and $f_\Delta$ are purely atomic, both with all atoms at positions $i/\Delta$, we use the 1-Wasserstein distance that becomes the maximal difference between the size of their atoms in our case:
The size of the atom at $i/\Delta$ for  $ \widehat{f_{n, \Delta}}$ is $\sum_{k=1}^n \frac{1}{n}{\bf 1}_{\{(i-1)/\Delta < X_k \leq i/\Delta\}}$. The size of the atom at $i/\Delta$ for  $f_{\Delta}$ is $\int_{(i-1)^+/\Delta}^{i/\Delta} f_\Delta(s) ds$.
Here, the 1-Wasserstein distance is
\[ d(\widehat{f_{n, \Delta}}, f_\Delta)  := \max_{i=1,...,\Delta} \left\vert \sum_{k=1}^n \frac{1}{n} {\bf 1}_{\{(i-1)/\Delta < X_k \leq i/\Delta\}} - \int_{(i-1)^+/\Delta}^{i/\Delta} f_\Delta(s) ds \right\vert.\]
We have the following lemma:

\begin{lemma}\label{lem:sample}
 %   \red{la norme n'est pas définie, et il y a confusion entre $f$ densité ou pas densité.}
    Let $\Delta \in \N$. Denoting $A := d(\widehat{f_{n, \Delta}}, f_\Delta)$. Then for all $i \in [0, \Delta]$:
    \begin{equation}\label{eq:sample}
        \left(\frac{1 - \frac{A\Delta}{L}}{1+\frac{A \Delta}{\ell}}\right) \nu_\Delta(\frac{i}{\Delta})
    \leq \widehat{\nu_{n, \Delta}}(\frac{i}{\Delta}) \leq \left(\frac{1 + \frac{A\Delta}{\ell}}{1 -
        \frac{A\Delta}{L}}\right) \nu_\Delta(\frac{i}{\Delta}).
  \end{equation}
\end{lemma}

\begin{proof}
    We first check that $\forall i \in [1, \Delta]$, we have:

    $$\frac{\ell}{\Delta} \leq \int_{\frac{i}{\Delta}}^{\frac{i+1}{\Delta}} f_\Delta(s) ds \leq \frac{L}{\Delta}.$$

    \textbf{Left part:} we simply use the fact that $f \geq \ell$.

    \textbf{Right part:} $f$ is $L$-lipschitz continuous and has support $[0,
    1]$. Thus $f \leq L$.
Let us now prove the right inequality in \eqref{eq:sample}
    \begin{align*}
        \widehat{\nu_{n, \Delta}}(\frac{i}{\Delta}) 
        &= \sup_{t' > \frac{i}{\Delta}} \frac{\int_{\frac{i}{\Delta}}^{t'} \widehat{f_{n, \Delta}}(s)
            ds}{\int_{\frac{i}{\Delta}}^{t'}
        \widehat{\bar{F}_{n, \Delta}}(s) ds} \\
        &= \max_{j > i} \frac{\int_{i /\Delta}^{j/\Delta} \widehat{f_{n, \Delta}}(s)
            ds}{\int_{i/\Delta}^{j/\Delta}
        \widehat{\bar{F}_{n, \Delta}}(s) ds} \\
        &= \max_{j > i} \frac{\sum_{k=i}^{j-1} \int_{k/\Delta}^{(k+1)/\Delta}
            \widehat{f_{n, \Delta}}(s)
            ds}{\sum_{k=i}^{j-1}  \sum_{l=k}^{\Delta-1}
                \int_{l/\Delta}^{(l+1)/\Delta}
            \widehat{f_{n, \Delta}}(s) ds} \\
        &\leq \max_{j > i} \frac{\sum_{k=i}^{j-1} \int_{k/\Delta}^{(k+1)/\Delta}
            (f_\Delta(s) +
            A) ds}{\sum_{k=i}^{j-1} \sum_{l=k}^{\Delta-1}
                \int_{l/\Delta}^{(l+1)/\Delta}
        (f_\Delta(s) - A) ds} \\
        &\leq \max_{j > i} \frac{\left( \sum_{k=i}^{j-1} \int_{k/\Delta}^{(k+1)/\Delta}
            f_\Delta(s) ds \right) (1+ \frac{A\Delta}{\ell})}{\left( \sum_{k=i}^{j-1} \Delta
            \sum_{l=k}^{\Delta}
        \int_{k/\Delta}^{(k+1)/\Delta} f_\Delta(s) ds \right)(1 - \frac{A\Delta}{L})} \\
        &\leq \nu_\Delta(\frac{i}{\Delta}) \frac{1 + \frac{A \Delta}{\ell}}{1 -
          \frac{A\Delta}{L}}.
    \end{align*}
  {Similarly : }
        \begin{align*}
        \widehat{\nu_{n, \Delta}}(\frac{i}{\Delta}) 
        &\geq \max_{j > i} \frac{\sum_{k=i}^{j-1} \int_{k/\Delta}^{(k+1)/\Delta}
            (f_\Delta(s) -
            A) ds}{\sum_{k=i}^{j-1} \Delta \sum_{l=k}^{\Delta-1}
                \int_{k/\Delta}^{(k+1)/\Delta}
        (f_\Delta(s) + A) ds} \\
        &\geq \max_{j > i} \frac{\left(\sum_{k=i}^j \int_{k\Delta}^{(k+1)\Delta}
            f_\Delta(s) ds \right)(1
            - \frac{A \Delta}{L})}{\sum_{k=i}^{j-1} 
            \sum_{l=k}^{\Delta-1}
        \int_{k/\Delta}^{(k+1)/\Delta} f_\Delta(s) ds (1 + \frac{A\Delta}{\ell})  } \\
        &\geq \nu_\Delta(\frac{i}{\Delta}) \frac{1 - \frac{A \Delta}{L}}{ + \frac{A \Delta }{\ell}}.
    \end{align*}
\end{proof}

%\begin{corollary}
%    Suppose that $\lim_{n \rightarrow \infty} \Delta \lVert \hat{f}_{N,n} - f_N
%    \rVert_{\infty} = 0$, then $\lim_{n \rightarrow \infty} \hat{\nu}_{N,n} = \nu_{N}$.
%\end{corollary}

We are now ready to compare the index of the shifted sampled distribution with the real ones.
To do so, we need  $n$ to be large enough compared with  $\Delta$.

\begin{lemma}\label{lem:probaSample}
  Let $n \in \N$ be the number of samples. Assume $n \geq \Delta^{2 +4\alpha}$, where $1/2 > \alpha >0$. Under the foregoing assumptions, 
  with probability greater than $1-2 \Delta \exp(-2n^\alpha)$ we get for all $0<i\leq \Delta$,
   $$\frac{1 -
    \frac{n^{-\alpha (1-2\alpha)}}{L}}{1+\frac{n^{-\alpha (1-2\alpha)}}{\ell}}
    \nu_\Delta(i/\Delta)
    \leq \widehat{\nu_{n, \Delta}}(i/\Delta) \leq \frac{1 + \frac{n^{-\alpha (1-2\alpha)}}{\ell}}{1 -
    \frac{n^{-\alpha (1-2\alpha)}}{L}} \nu_\Delta(i/\Delta).$$
\end{lemma}

\begin{proof}
  
  The proof is based on Hoeffding inequality.
  
Let $X_1,\ldots,X_n$ be $n$ iid samples from the distribution with density $f$.
Let $Y^i_j = {\bf 1}(i/\Delta < X_j \leq (i+1)/\Delta)$.
For all $i$, $(Y^i_j)_{j=1...n}$ are iid random variables with values in $\{0,1\}$ and  mean $\int_{i/\Delta}^{(i+1)/\Delta} f(s) ds$.
We also denote the sum $S_n^i :=\sum_{j=1}^n Y^i_j$

Now by definition of the two shifted densities, for all $A$:
    \begin{align*}
        \mathbb{P}(d(\widehat{f_{n, \Delta}},f_\Delta )  \geq A) & = 
                                                                                     \sup_{i=0...\Delta-1} \mathbb{P} ( | (S_n^i /n- \int_{i/\Delta}^{(i+1)/\Delta} f(s) ds)| \geq A).
    \end{align*}
    Bounding the supremum by the sum, we get 
   \begin{align}
        \mathbb{P}(d(\widehat{f_{n, \Delta}},f_\Delta )  \geq A) & \leq
                                                                                     \sum_{i=0}^{\Delta-1} \mathbb{P} ( |(S_n^i /n- \int_{i/\Delta}^{(i+1)/\Delta} f(s) ds)  | \geq A) \label{eq1} \\
                                                                                   &\leq   \sum_{i=0}^{\Delta-1}  2 \exp(-2nA^2) \label{eq:Hoeffding} \\
      &=   2\Delta  \exp(-2nA^2),
         \end{align}
         where \eqref{eq:Hoeffding} uses Hoeffding inequality.

% We recall Heoffding's inequality:

% \begin{lemma}[Hoeffding inequality]
%     Let $X_1, \cdots, X_n$ be independent random variables such that $a_i \leq X_i
%     \leq b_i$ almost surely. Consider the sum of these random variables,
%     $S_n = X_1 + \cdots + X_n$.

%     Then Hoeffding's theorem states that, for all $t > 0$,
%     $$\mathbb{P}(S_n - \mathbb{E}(S_n) \geq t) \leq
%     \exp(-\frac{2t^2}{\sum_{i=1}^n (b_i - a_i)^2})$$
% \end{lemma}

% \begin{proof}
%     In our case, we have variables $X_{i,j}$ = "the sample $j$ lies between
%     $i/\Delta$ and $(i+1)/\Delta$, thus $a_i = 0$ and $b_i = 1$.

%     \begin{align*}
%         \mathbb{P}(\lVert \hat{f}_{n, \Delta} - f_\Delta \rVert_\infty \geq A)
%         &\leq \sum_{i=0}^{\Delta-1} \mathbb{P}(\sup_{i\Delta \leq t \leq
%         (i+1)\Delta}(\hat{f}_{n, \Delta}(t) - f_\Delta(t)) \geq  A) \\
%         &\leq \Delta \int_{i\Delta}^{(i+1)\Delta} \hat{f}_{n,
%         \Delta}(s) - f_\Delta(s) ds \geq
%         A) \\
%         &\leq \Delta \exp(-2n A^2) \\
%     \end{align*}

    We now tune the variables:  $A := n^{-\frac{1}{2} +
    \frac{\alpha}{2}}$ and $n \geq \Delta^{2+4\alpha}$. We get:

    \begin{align*}
        \mathbb{P}(d(\widehat{f_{n, \Delta}},f_\Delta ) \geq n^{-\frac{1}{2} +
        \frac{\alpha}{2}})
        &\leq 2\Delta \exp(-2n(n^{\frac{-1 +
        \alpha}{2}})^2) \\
        &\leq 2 \delta  \exp(-2n(n^{\frac{-1 +
        \alpha}{2}})^2) \\
        &\leq 2 \delta  \exp(-2n^\alpha) \\
    \end{align*}

    Using Lemma \ref{lem:sample}, we get that with probability greater than $1-2\Delta \exp(-2n^\alpha)$,
$$\frac{1 -
    \frac{n^{-\alpha (1-2\alpha)}}{L}}{1+\frac{n^{-\alpha (1-2\alpha)}}{\ell}}
    \nu_\Delta(i/\Delta)
    \leq \widehat{\nu_{n, \Delta}}(i/\Delta) \leq \frac{1 + \frac{n^{-\alpha (1-2\alpha)}}{\ell}}{1 -
    \frac{n^{-\alpha (1-2\alpha)}}{L}} \nu_\Delta(i/\Delta).$$
\end{proof}

\section{The (shifted) empirical Gittins index policy}

\subsection{Performance of the shifted empirical index policy}
\label{ssec:perf}

The fact that the shifted empirical index is close to the original one does not imply directly
that the performance of the policy based on this index is close to the performance of the policy using the real index because the choice of the next job is based on the comparison between indices that is not continuous in their  values.
For that we first construct a preemptive set $\XP^{\Delta,\epsilon} := \{ i/\Delta, s.t. i/\Delta < 1-\epsilon\}$.
In other words, preemption is not allowed if $x$ is above $1-\epsilon$ and if $x$ is not an integer multiple of $1/\Delta$.

The following result from \cite{ScullyPhD} (Theorem 16.5) gives a robustness guarantee of the performance of a policy based on indexes, that may not be the Gittins indexes:  \emph{Consider a stable M/G/1 queue  ($\lambda \E(X) <1$).
    Let $\nu$ be the Gittins index function derived from the  job
    size distribution. Assume there exist $\Gamma >0$ and a  priority policy $\pi$ using an index
    function $\nu'$ satisfying for all $t$,
    $$\frac{1}{\gamma} \nu(t) \leq \nu'(t) \leq \gamma \nu(t).$$
    Then the policy $\pi$ induces  a mean response time within a factor of $\gamma^2$ of the optimal:
$$\mathbb{E}(T_\pi) \leq \gamma^2 \mathbb{E}(T_{\textsc{Gittins}}).$$}

This result combined with our previous lemmas allow us to state our first main result.

\begin{theorem}\label{th:main}
    Consider a stable M/G/1 queue in which the scheduler does not observe individual
    job sizes. Assume that $f$ (the density of job size distributions) satisfies assumption \ref{ass:1}. Let $n$ be the number of samples, $\epsilon >0$, $1/2 > \alpha >0$ be two constants and let  $n \geq  \Delta^{2 +4\alpha} $.
     Define
      \[ \gamma_n  = \frac{1 + \frac{n^{-\alpha (1-2\alpha)}}{\ell}}{1 -
            \frac{n^{-\alpha (1-2\alpha)}}{L}} \frac{1 + \frac{L}{\Delta \epsilon
              \ell}}{1 - \frac{L}{\Delta \epsilon \ell}} \quad\mbox{ and }  \quad p_n = 2\Delta \exp(-2n\alpha),\]
              (note that $\gamma_n $ goes to 1  and $p_n$ goes to 0 when $\Delta$ goes  to $+\infty$).
              Let $\pi_{\textsc{shifted-samples}}$ be the priority policy using the  shifted empirical index function $\widehat{\nu_{n,\Delta}}$ and preemption set $\XP^{\Delta,\epsilon}$. Then, by denoting $K$ a bound on the expected response time of any work-conserving policy,
     $$   \mathbb{E}(T_{\textsc{shifted-samples}}) \leq \gamma_n^2
        \mathbb{E}(T_{\textsc{Gittins}}) + p_n K.$$
\end{theorem}

\begin{proof}
    From Lemma \ref{lem:shifted} we get for all states in the preemptive set $\XP^{\Delta.\epsilon}$,

    $$\nu(\frac{i}{\Delta}) \frac{1 - \frac{L}{\Delta \ell \epsilon}}{1 + \frac{L}{\Delta
        \epsilon \ell}}\leq \nu_\Delta(\frac{i}{\Delta}) \leq \nu(\frac{i}{\Delta}).$$

    Now, let us  define the good event $\mathcal{E}_n = \{ \lVert \widehat{f_{n, \Delta}} - f_\Delta \rVert_\infty\leq n^{-1/2 +\alpha/2} \}$. Lemma \ref{lem:probaSample} says the $\pr(\mathcal{E}_n ) \geq 1-p_n$.

    Under the good event $\mathcal{E}_n$, 

    $$\frac{1 -
    \frac{n^{-\frac{\epsilon}{2}}}{L}}{1+\frac{n^{-\alpha (1-2\alpha)}}{\ell}}
    \nu_\Delta(i/\Delta)
    \leq \hat{\nu}_{n, \Delta}(i/\Delta) \leq \frac{1 + \frac{n^{-\alpha (1-2\alpha)}}{\ell}}{1 -
    \frac{n^{-\alpha (1-2\alpha)}}{L}} \nu_\Delta(i/\Delta).$$

    Thus setting  $\gamma = \frac{1 + \frac{n^{-\alpha (1-2\alpha)}}{\ell}}{1 -
            \frac{n^{-\alpha (1-2\alpha)}}{L}} \frac{1 + \frac{L}{\Delta  \epsilon
              \ell}}{1 - \frac{L}{\Delta  \epsilon \ell}}$  we have
          \begin{equation} \label{eq:both}
            \frac{1}{\gamma} \hat{\nu}_{n} (i/\Delta) \leq \nu(i/\Delta)
    \leq \gamma \hat{\nu}_n(i/\Delta). 
  \end{equation}

    We now use $K$ a bound on the response time of any work-conserving scheduling policy in a M/G/1 queue and the robusteness theorem of \cite{ScullyPhD} previously given.  to conclude:

    \begin{align*}
      \E(T_{\textsc{shifted-samples}}) &=  \E(T_{\textsc{shifted-samples}} \one_{\mathcal{E}_n} + T_{\textsc{shifted-samples}}
                                        \one_{\mathcal{E}_n^{\mathsf{c}}})\\
      & \leq \E(T_{\textsc{shifted-samples}} \one_{\mathcal{E}_n} ) +  K \E( \one_{\mathcal{E}_n^{\mathsf{c}}}) \\
                                      & \leq  \gamma_n^2 \E(T_{\textsc{Gittins}}) + p_n K,
    \end{align*}
    where first inequality follows from the definition of $K$ and the last one uses Eq. \eqref{eq:both} and Theorem 16.5 in \cite{ScullyPhD}.
      
    % that   $  \mathbb{E}(T_{\textsc{shifted-sample}}) \leq \gamma_n^2
    %     \mathbb{E}(T_{\textsc{Gittins}}).$
    %   With probability $p_n$,  $\mathbb{E}(T_{\textsc{Sampled}}) $ cannot be bounded by   $ \mathbb{E}(T_{\textsc{Gittins}}).$ However, it is obviously bounded by $K$, the expected response time of the worse possible scheduling policy in a M/G/1 queue.

        % For now, let us assume that $K$ is finite. We get the final result,     $\mathbb{E}(T_{\textsc{Sampled}}) \leq \gamma_n^2 \mathbb{E}(T_{\textsc{Gittins}}) + p_n K.$

    The remaining point is to  find an explicit  bound on the mean response time in a stable M/G/1 queue, under any work-conserving policy. The following  is pure queueing theory and can be seen as a more or less direct consequence of the study of the M/G/1 queue in \cite{Mor}. The following derivation  was given to us by Ziv Scully.
    Everything  holds as soon as (i) the size distribution has a finite mean $1/\mu$ smaller than the mean interarrival time $1/\lambda$ ($\rho :=\lambda/\mu < 1$) and (ii) it has a finite variance (which  is true here since $f$ has a finite support).
    Let us consider the workload $W(t)$ in the M/G/1 queue at time $t$.  $W(t)$ is the same under of any work conserving policy $\pi$, and so are  the busy periods (periods where $W(t) > 0$).
    Therefore, the response time of any policy $\pi$ is bounded by the length of a busy period $B(W_a+X)$ starting with $W_a$ (the work present in the system when a job arrives)  plus $X$ (the size of the job).
    This means $\E(T_\pi) \leq  \E(B(W_a+X)) = \E(W_a + X)/ (1-\rho)$.
    
    Using  PASTA, the workload at job arrivals $W_a$ has the same law as  the time-average worload $W$. The Pollaczek–Khinchine formula gives 
    $\E(W) = \frac{\lambda \E(X^2)}{2(1-\rho)}$.
    
    Putting everything together,
    $\E(T_\pi) \leq  \frac{\E(X)}{1-\rho} + \frac{\lambda\E(X^2)}{2(1-\rho)^2} =:K$.
   
  \end{proof}

  \subsection{Empty non-preemptive set}
  \label{ssec:no-preemption}
 So far, we considered the case where the preemptive set is restricted to $\XP^{\Delta,\epsilon}$.
  As announced  at the beginning, we now consider the case where the set $\XP = \cX$ and $\XNP = \emptyset$.
  The robusteness theorem in  \cite{ScullyPhD} is not directly applicable here because outside states in   $\XP^{\Delta,\epsilon} $, 
  the indices are quite different, as seen in  Figure \ref{fig:index}:  the empirical index approaches infinity while the true index remains bounded as the age approaches $i/\Delta$ from below.

  To deal with 
  $\XP = \cX$, we will let the discretization step $\Delta$ grow with $n$, instead of keeping it fixed and we also let $\epsilon$ go to $0$.
  To comply with the assumptions in Theorem \ref{th:main}, we typically choose $\Delta$ to grow roughly  like $ n^{2}$ and $\epsilon$ grow like $\sqrt{\Delta}$.
  The next theorem gives  the asymptotic behavior of our empirical Gittins policy without any restrictions on preemptions.

  %Let us denote by $c(n,\Delta)$ the average response time  under  the empirical index policy with  $\XP^\Delta  = \{ i/\Delta, i\in \N\}$.

  % Theorem \ref{th:} says that  $c(n,\Delta)$  converges to $c^*(\Delta)$, the optimal response time  with  $\XP^\Delta  = \{ i/\Delta, i\in \N\}$, for all $\Delta \in N$.

\begin{theorem}
    Consider a stable M/G/1 queue in which the scheduler does not observe individual
    job sizes. Assume that $f$ (the density of job size distributions) satisfies assumption \ref{ass:1}. Let $n$ be the number of samples, $\epsilon >0$, $1/2 > \alpha >0$ be the two constants as before and let  $ \Delta = n^{2/(1+8\alpha)}$ and $\epsilon = 1/\sqrt{\Delta}$. Under preemption set $\XP = \cX$, as $n$ goes to infinity, 
     $  \mathbb{E}(T_{\textsc{shifted-samples}})$ converges to 
       $ \mathbb{E}(T_{\textsc{Gittins}})$.
\end{theorem}

\begin{proof}
  Let us denote by $\pi(\nu)$ the scheduling policy index the true index $\nu$.

  The first step is to remove the preemption restriction $x > 1-\epsilon$.
  Under Assumption \ref{ass:1} on the size distribution, the rank function  is continuous and remains bounded away from $0$ in $[0,1)$. Furthermore, $\nu(x) > 1/1-x$ for all $x\in (0,1)$. Therefore, if $\epsilon$ is small enough (smaller than a critcal value $\epsilon_c$), then $\nu(x) > \nu(y)$ for all $x> 1-\epsilon_c$ and all $y < 1-\epsilon_c$. This implies that if a job reaches age  $x > 1-\epsilon_c$ then the job will not be preempted before completion even is preemption is allowed, so that restricting the preemption after $1-
  \epsilon_c$ has no impact on the execution of these large jobs as long as $\epsilon < \epsilon_c$.

    In the second step, we now focus on one  busy period of the M/G/1 queue (remember that this busy period does not depend on the work-conserving policy used to scedule jobs).
    under  $\pi(\nu)$, the busy period is decomposed into a finite number of time intervals, corresponding to which jobs are executed at any time (this may include time intervals with processor sharing.
    The end points of these intervals are of two types: preemptions and completions of a job.
    
  Now, let us construct a new scheduling policy $\pi_\Delta$ that mimics $\pi(\nu)$ but  shifts  all the preemption end points to next point in $ \{ i/\Delta, i\in \N\}$. As for the completion end points, they are moved to be the new completion times under the new policy. This policy is not work conserving (it may have to wait to start a new job after a completion, if this completion happens  before the conpletion under the Gittins policy) and it is not an index policy
  The policy $\pi^\Delta$ only preempts at points in  $\XP^\Delta$ so its response time is greater than the response time of the optimal polcy under   $\XP^\Delta$, namely $\pi(\nu_\Delta)$. On the other hand, its response time is close to the response time of $\pi(\nu)$, missing the mark on the response time for each job by only a gap smaller than $m^2 /\Delta$, where $m$ is the number of  intervals in the busy period for the job. On average, by explicitly mentionning the preemption set, 
  
  \begin{equation}
    \mathbb{E}(T_{\textsc{Gittins},\cX})  + M_2/\Delta \geq \mathbb{E}(T_{\textsc{Gittins},\XP^\Delta}) \geq 
    \mathbb{E}(T_{\textsc{Gittins},\cX}), \label{eq:compaPreempt}
    \end{equation}
where $M_2$ is the expectation of $m^2$ with $m$ the number of intervals in a busy period under  $\pi(\nu)$.
Finally, when $\Delta$ goes to infinity, this implies that the performance of the policy under discrete preemption times  converges to the perfomance  of the  policy $\pi(\nu)$ under all preemption times.
Finally, since  the performance of the policy using empirical estimates is close to the  performance of the Gittins policy under discrete preemptions when $n$ grows, this implies together with Eq. \eqref{eq:compaPreempt}, that it also converges to the performance of the Gittins policy when $\XP = \cX$.
\end{proof}
  
\section{Numerical comparisons}
\label{sec:numerical}

\subsection{Trajectorial comparison with the Gittins index policy}

In our first experiment, we compare a trajectory of two scheduling policies the first uses the true Gittins index $\nu$, while the second policy uses the shifted empirical index $\widehat{\nu_{n,\Delta}}$, with $n = 1000$ and $\Delta = 31$. For this experiment, we generate the size of the jobs according to a Pareto($2$) conditioned on being smaller than $5$. The CDF as well as the Gittins index for this particular distribution are displayed in Figure~\ref{fig:cdf}(b) and Figure~\ref{fig:index}(b).  In this Figure, we use the same set of 10 jobs, starting with an empty system. The jobs are generated using a Poisson arrival process with rate $\lambda = 3$. The size of each job is generated according to a Pareto(2) distribution conditioned on being less than $5$. In Figure~\ref{fig:traj}, we couple the arrival time and the sizes of the jobs and compare the execution of the schedule obtain when using the true Gittins index (in Panel~(a)) and the shifted empirical Gittins (in Panel~(b)). In all cases, the gray areas corresponds to the time when the jobs are present in the system but not executed. 

When running our simulation, we assume that only one job can be served at a time and we set a minimal execution time of $dt=0.01$. This means that when a job is scheduled to run, it cannot be interrupted during the next $dt$ time units. In Figure~\ref{fig:traj}(a), we observe that the true Gittins policy would sometimes like to schedule up to four jobs at the same time (for instance from time $2$ to $2.5$), is a processor-sharing fashion. Instead, our schedule alternates between the four jobs (based on Gittins index).  When using the shifted empirical Gittins (Panel~(b)), the situation is different because each job is scheduled for at least $5/\Delta\approx0.16$ unit of time. Hence, in this case the alternance between the various job sizes are less frequent. This figure illustrates that the schedule used by the two policies can be quite different, 

    % The Figure \ref{fig:traj} shows the resulting execution sequence of the jobs.
    % One can spot some differences between both schedule, in particular, under
    % the true Gittins policy, some jobs have the same index and are executed
    % together (in PS fashion): For example jobs 0, 1 and 2, are executed
    % together under $\nu$ while execution is sequential under
    % $\widehat{\nu_{n,\Delta}}$, albeit quite interleaved.
    % These differences do not result in big differences for response times, as shown in the following Table.

%     \begin{center}
%         \begin{tabular}{| c | c | c | c | c | c | c | c | c |}
%         \hline
%         Job number   & 0 & 1 & 2 & 3 & 4 & 5 & 6 & 7 \\
%         Response time Gittins & 1.59 & 0.25 & 0.31 & 0.77 & 0.18 & 0.37 & 0.02 &
%         0.03 \\
%         Response time shifted samples Gittins & 1.51 & 0.14 & 0.18 & 1.46 & 0.24
%                                               & 0.21 & 0.06 & 0.03 \\
%         \hline
%     \end{tabular}
%   \end{center}

\begin{figure}[htbp]
    \centering
    \begin{subfigure}[b]{0.47\textwidth}
        \includegraphics[width=\textwidth]{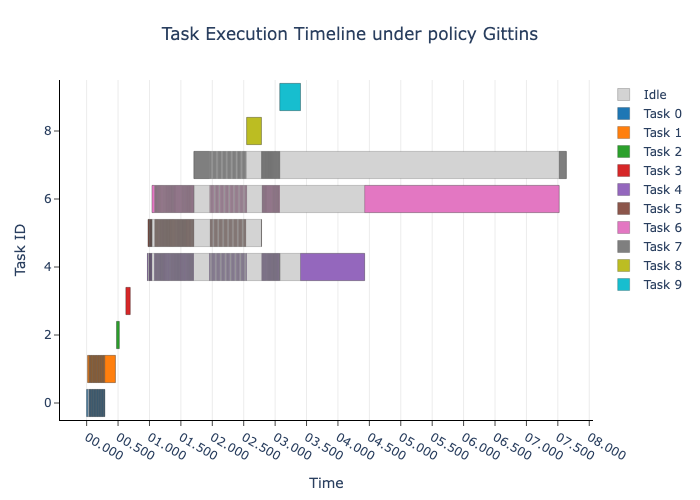}
        \caption{True Gittins policy.}
        \label{fig:image1}
    \end{subfigure}
    \hfill
    \begin{subfigure}[b]{0.47\textwidth}
        \includegraphics[width=\textwidth]{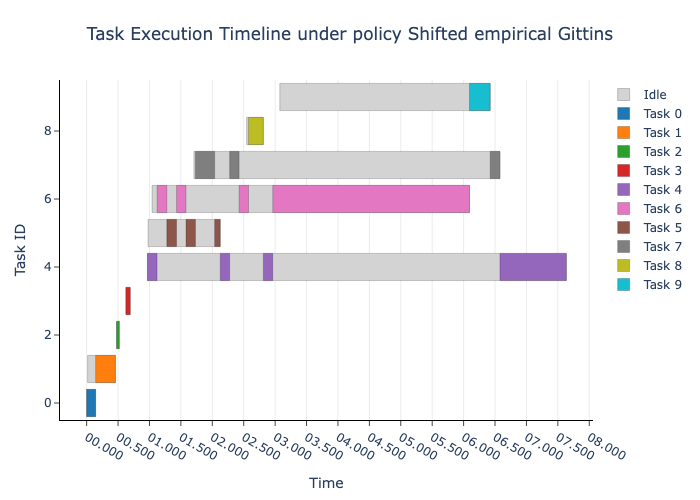}
        \caption{Shifted empirical Gittins ($n=1000$ samples).}
        \label{fig:image2}
    \end{subfigure}
    \caption{Task execution timeline for two scheduling policies (Gittins and shifted empirical Gittins).}
    \label{fig:traj}
  \end{figure}

\subsection{Average response time}

Our second set of experiments compares the average response times under several alternative policies. We compare three main policies:
  \begin{itemize}
    \item The true Gittins index $\nu$ (which provides the optimal response time)
    \item The shifted empirical index $\widehat{\nu_{n,\Delta}}$ (the policy studied in this paper that is asymptotically optimal as $n$ goes to infinity).
    \item The empirical Gittins index $\widehat{\nu_n}$, obtained by computing the Gittins index from the empirical distribution and whose performance guarantees are provided in \cite{ramakrishna2026empirical}.
  \end{itemize}
  To make a more complete comparison, we also add two other heuristics to which we compare the performance of our indices:
  \begin{itemize}
    \item \emph{Gaussian KDE Gittins} -- An index policy where indexes are computed using the convolution distribution $f_G = \widehat{f_n} \star G$, where $G$ is a Gaussian kernel with a small variance.  The kernel method  is often used to estimate the true distribution $f$ using $n$ samples (see for example, \cite{Rudemo}). The idea of this heuristic is to smoothen the empirical distribution and avoid the ``spikes'' of the sampled Gittins. 
    \item \emph{FCFS} (first-come-first-served). We add this policy only as a baseline comparison. It is not expected to perform well if the job size has a decreasing failure rate but can perform well otherwise.  
%   %\item $\pi(\widehat{\nu_n})$ is the policy using the index $\widehat{\nu_n}$, the index of the empirical distribution (made of $n$ atoms).

% %  \item $\pi(\widehat{\nu^N_{n}})$ uses the index obtained by a different shifting procedure.
%     Instead of shifting each sample to the next discrete point on  a regular grid, samples are clustered (to their rights) by packets  of size  $N$.
%     More formely, all samples are ordered by their size and  the $i$th greatest sample $X_i$ is shifted to become equal to $X_{\lceil i/N \rceil N}$.

%   \item $\pi(\widehat{\nu_{n,{\min}}})$ uses the index  $\widehat{\nu_{n,{\min}}}(t) := \min_{u\in [t-h,t+h]} \widehat{\nu_n}(u)$, for a well chosen small $h$. % Note that this index function is pessimistic about the future failure rate of the job. This property is shared by the shifted empirical index. \red{Est-ce qu'il faudrait vraiment mettre ça?}
   
  \end{itemize}
  
%   For reference purposes we also display  the performance of the policy using the true index $\nu$ (that knows the distribution), the performance of $SRPT$ (that knows the sizes) and two oblivious policies, namely FCFS (First Come First Served) and LCFS (Last come first Served).

\begin{figure}[htbp]
    \centering
    \begin{tabular}{@{}c@{}c@{}}
        \begin{subfigure}[b]{.5\textwidth}
            \includegraphics[width=\textwidth]{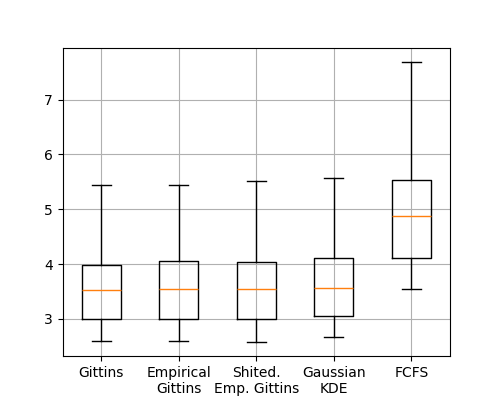}
            \caption{Distribution $\beta_{12, 50} + \beta_{10, 2}$}
            \label{fig:image4}
        \end{subfigure}
        &
        \begin{subfigure}[b]{.5\textwidth}
            \includegraphics[width=\textwidth]{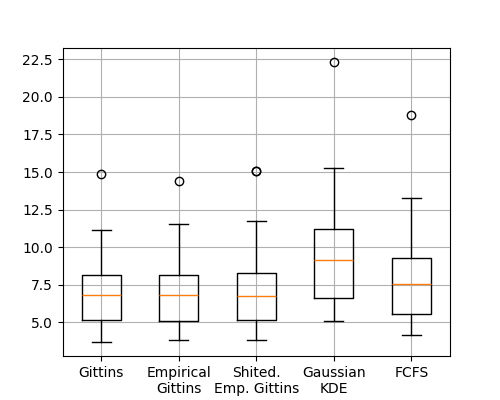}
            \caption{Pareto(2) distribution.}
            \label{fig:compa}
        \end{subfigure}
    \end{tabular}
     \hfill
     
     \caption{Boxplot of average services times for different distributions.}
     \label{fig:side_by_side2}
\end{figure}

The experiment shows that the performance of the shifted empirical Gittins and of the empirical Gittins match very closely the performance of the true Gittins index policy. For the shifted empirical Gittins, this illustrates our result that shows that the empirical Gittins policy is asymptotically optimal. For empirical Gittins, this illustrate the result of \cite{ramakrishna2026empirical}.  This results should be compared to the case of the Gaussian KDE Gittins and of FCFS that perform significantly worse.

\bibliographystyle{plain}
\bibliography{biblio}

\end{document}